\documentclass[12pt,leqno,a4paper]{article}
\usepackage{color}
\usepackage{amsfonts}
\usepackage{amsmath, amsfonts, amsthm, amssymb, amscd}
\usepackage[mathcal]{euscript}
\usepackage{mathrsfs}
\usepackage{dsfont}
\usepackage{ulem}
\usepackage{bbm}
\usepackage{graphicx}

\usepackage{mathtools}
\usepackage{slashed}
\usepackage{lineno} 
\usepackage{verbatim}
\usepackage{appendix}

\usepackage{tikz-cd}

\newtheorem{theorem}{Theorem}[section]
\newtheorem{proposition}[theorem]{Proposition}
\newtheorem{lemma}[theorem]{Lemma}

\newtheorem{definition}[theorem]{Definition}
\newtheorem{remark}[theorem]{Remark}

\numberwithin{equation}{section}

\usepackage{multirow}
\usepackage{tabularx}
\usepackage{lscape}
\newcommand \Rbb {\mathbb{R}}

\newcommand \Xcal{\mathcal{X}}

\newcommand \del \partial

\newcommand {\vep}{\varepsilon}

\makeatletter
\def\hlinew#1{%
  \noalign{\ifnum0=`}\fi\hrule \@height #1 \futurelet
   \reserved@a\@xhline}
\makeatother

\title{Standing Sphere Blow-up Solutions to The Nonlinear Heat Equation
}

\begin{document}
\maketitle
\centerline{Senhao Duan$^{(1)}$, Nejla Nouaili$^{(1)}$ and Hatem Zaag$^{(2)}$} 
\medskip
{\footnotesize
	\centerline{ $^{(1)}$ CEREMADE, Universit\'e Paris Dauphine, Paris Sciences et Lettres, France}
 \centerline{ $^{(2)}$Universit\'e Sorbonne Paris Nord,
			LAGA, CNRS (UMR 7539), F-93430, Villetaneuse, France.}
	}

     \begin{abstract}
    In this paper, we construct a singular standing-sphere solution to the nonlinear heat equation in the radial case. We give rigorous proof for the existence of such a blow-up solution in finite time. This result was predicted numerically by Baruch, Fibich, and Gavish \cite{BFGpd10}. We also prove the stability of these dynamics among radially symmetric solutions.
\end{abstract}

%%%%%%%%%%%%%%%%%%%%%%%%%%%%%%%%%%%%%%%%%%%%%%%%%%%%%%
%                   6. BODY
%%%%%%%%%%%%%%%%%%%%%%%%%%%%%%%%%%%%%%%%%%%%%%%%%%%%%%

% Only the first word and proper nouns of section titles should be capitalized.
% The title of section 1:

\section{Introduction}
In this paper, we consider the following nonlinear heat equation
\begin{equation}\label{nlh-introduction}
	\left \{
	\begin{array}{lll}
		\partial_t u&=&\Delta u+|u|^{p-1}u,\\
		u(.,0)&=&u_0 \in L^{\infty}(\Rbb^d),
	\end{array}
	\right.
\end{equation}
where $u(t):x\in \Rbb^d \to u(x,t)\in \mathbb{R}$ and $p>1$.
Equation \eqref{nlh-introduction}
is considered as a model for many physical situations such as heat transfer, combustion theory and thermal explosion. (see more in Kapila \cite{KSJAM80}, Kassoy and Poland \cite{KPsiam80} and Bebernes and Eberly \cite{BEbook89}).

The local Cauchy problem for equation \eqref{nlh-introduction} can be solved within ${L^\infty(\Rbb^d)}$. Furthermore, it can be shown that the solution $u(t)$ exists either in the interval $[0, +\infty)$ or within $[0, T)$ where $T < +\infty$.  In the latter case, $u$ undergoes a finite-time blow-up. $T$ is then called blow-up time, indicating that 
$$\lim_{t\to T}||u(t)||_{L^\infty}=+\infty.$$ 
Moreover, a point $x_0$ is called a blow-up point if there are  sequences $\{x_n\}\to x_0$ and $\{t_n\}\to T$, such that $\limsup_{n\to\infty}|u(t_n,x_n)|=+\infty$.

Despite the extensive research conducted on these equations over the past four decades, it is crucial to acknowledge that no single review can comprehensively cover all aspects. In this context, our attention is directed toward the development of solutions displaying a distinct blow-up behavior. Consequently, our references will be limited to prior work within this scope.
Interested readers may refer to \cite{QSbook07} for comprehensive  research on equation \eqref{nlh-introduction}.
The pioneering work by Giga and Kohn \cite{GKcpam85} and \cite{GMSmmas04} yielded the first insight into the asymptotics of blowup. 
\begin{comment}
	By introducing the following self-similar variables:
	\begin{equation}
		\begin{aligned}
			y&=\frac{x-a}{\sqrt{T-t}},\\
			s&=-\log(T-t),\\
			W(y,s)&=(T-t)^{\frac{1}{p-1}}U
			(x,t)\\
		\end{aligned}
	\end{equation}
\end{comment}
They established that, up to changing $u$ by $-u$, for each $K>0$, the following holds: 

\begin{equation}\label{eq:GK decay of w-k}
	\lim_{{t \to T}} \sup_{{|x| \leq K\sqrt{T-t}}} \left| (T - t)^{\frac{1}{p-1}} u(x, t) - \kappa \right| = 0, \mbox{ with $\kappa= (p-1)^{-\frac{1}{p-1}} $.}
\end{equation}
Based on a numerical analysis conducted by Berger and Kohn \cite{BKcpam88}, it was hypothesized that if the decay pattern is non-exponential, the solution $u$ to equation \eqref{nlh-introduction} would converge toward a specific universal profile, denoted as $f(z)$. 
Extensive literature is devoted to the blow-up profile for the NLH equation; see V\'elazquez \cite{VELcpde92}, \cite{VELtams93}, \cite{VELiumj93}, and Zaag \cite{ZAAaihp02}, \cite{ZAAcmp02} for partial results.
In one-space dimension, given $a$ a blow-up point, these are the situations:
\begin{enumerate}
	\item  $(T-t)^{\frac{1}{p-1}}u(x,t) \equiv \kappa$.
	\item either 
	\begin{equation}\label{eq: profile generic}
		\sup_{|x-a|\leq K\sqrt{(T-t)\log(T-t)}}\left| (T-t)^{\frac{1}{p-1}}u(t,x)-f\left(\frac{x-a}{\sqrt{(T-t)|\log(T-t)|}}\right)\right|\to 0,
	\end{equation}
	\item or for some $m\in \mathbb{N}$, $m\geq2$, and $C_m>0$
	\begin{equation}\label{profile-flatt}
		\sup_{|x-a|\leq K(T-t)^{1/2m}}\left|(T-t)^{\frac{1}{p-1}}u(x,t)-f_m\left( \frac{C_m(x-a)}{(T-t)^{1/2m}}\right)\right|\to 0,
	\end{equation}
	as $t\to T$, for any $K>0$, where 
	\begin{equation}\label{profile-generic-flat}
		\begin{aligned}
			f(z)&=(p-1+b_0z^2)^{-\frac{1}{p-1}},\mbox{ with } b_0=\frac{(p-1)^2}{4p}\\
			f_m(z)&= (p-1+b|z|^{2m})^{-\frac{1}{p-1}}, \mbox{ with } b >0. \\
		\end{aligned}
	\end{equation}
\end{enumerate}
%% \textcolor{red}{Here, you choose to give the blow-up behavior in one space dimension, then, somehow you jump to radial equations, without mentioning anything about higher dimensions, outside the radial setting (there are mainly the works by Vel\'azquez. In my opinion, we should not mention all the details as in one space dimension. Simply say that some results are available in higher dimensions (see Vel\'azquez), however, they are not as complete. For that reason, we would like to focus in  this paper on radial solutions, and then, you mention Fibich et al.}\\
In the higher dimensional case, we would like to mention the works of Herrero and V\'elazquez \cite{HVcpde92} on the asymptotic behavior of the blow-up solution to equation \eqref{nlh-introduction} and  \cite{VELtams93} on the classification of such behavior.
%Following these behaviors, there are also interesting examples.
 One may refer to Nguyen and Zaag \cite{NZ2017} for constructed solutions showing a refinement of behavior \eqref{eq: profile generic} in 2 dimensions. In the supercritical case, we have an example of a single-point blowup with a degenerate profile given by Merle, Raph\"ael, and Szeftel \cite{MRSIMRNI20}. More recently, Merle and Zaag \cite{MZIMRN22} provided an example of a degenerate blow-up solution with a completely new blow-up profile, which is cross-shaped.

% obeying with a degenerate profile and an isolated blow-up point. }
%But up to our knowledge, no one has given a blow-up solution obeying \eqref{eq: profile generic} in radial symmetrical case.} For this reason, we would like to provide such a example.

We review here the question of the existence of blow-up solutions obeying \eqref{eq: profile generic} and \eqref{profile-flatt}. Let us first mention that in the one dimensional case, the question was positively
answered by Bricmont and Kupiainen in  \cite{BKcpam88}  (see also Herrero and Vel\'azquez \cite{HVdie92} for the
case \eqref{profile-flatt} with $d = 4$). Recently, Duong et al. \cite{DNZArxiv2022} revisited the construction of \textit{flat} profile given by $f_m$ in \eqref{profile-generic-flat} using modulation theory.
The methods used in \cite{BKcpam88} were enhanced afterward by Merle and Zaag \cite{MZdm97} using a more geometrical approach. Generally speaking, regarding the linearized equation. The proof relies on the understanding of the dynamics of the self-similar version of \eqref{nlh-introduction}  around the profile \eqref{eq: profile generic}. More precisely, they proceed in two steps:

\begin{enumerate}
\item Reduce the problem into a finite-dimensional one.
\item Solve the finite-dimensional problem with a topological shooting argument. 
\end{enumerate}

This powerful method is then applied to many other different fields. Interested readers are invited to see \cite{MZjfa08}
and \cite{DNZMAMS20} for an application in the complex Ginzburg-Landau equation and a more direct way to accomplish the first step in \cite{MZdm97}. Besides, D\'avila, Del Pino, and Wei \cite{DDWIM20} applied this method to deal with the formation of the singularities for harmonic map flow. Among the numerous results, the blow-up behavior of the solution to the nonlinear Schrodinger equation remarked by Raphael \cite{Rdm06} on a sphere spikes our interest. They have studied the existence and stability of a solution blowing up on a sphere to the $L^2$-supercritical nonlinear Schr\"odinger equation by Rapha\"el \cite{Rdm06}, and this result was extended to a higher dimension in the work of Rapha\"el and Szeftel \cite{RSCMP09}. 

Indeed, the linearized equation of \eqref{nlh-introduction} in radial coordinates (see below \eqref{eq: U })  presents an additional singularity at zero. So, we use ideas from \cite{Rdm06}, \cite{RSCMP09} and  \cite{MNZNon2016}.  More precisely, we divide our study on two differents regions, the blow-up region and the regular one; see Section  \ref{Section-formulation-of-the-problem} below for more details.

Let us note that numerical evidence provided by Baruch, Fibich, and Gavish indicates that the nonlinear heat equation \eqref{nlh-introduction} allows for singular standing sphere solutions in cases of radially symmetric solutions \cite{BFGpd10}. We therefore aim to study  \eqref{nlh-introduction} in radially symmetric settings for the existence of singular standing sphere solutions and analyze their behavior, with a similar idea to \cite{MZdm97}. 

In this paper, we prove the existence and stability of radially symmetric blowups along the unit sphere of $\Rbb^d$ for the nonlinear heat equation \eqref{nlh-introduction}. Hence, we propose our main theorem as follows.
\begin{comment}
\begin{theorem}
\label{th: Main theorem}
(Existence of a singular standing solution for \eqref{nlh-introduction} with Prescribed Profile).
There exists \(T > 0\) such that Eq. \eqref{nlh-introduction} has a solution \(u(r, t)\) in \(S \times [0, T)\) such that:
\begin{enumerate}
	\item the solution \(u\) blows up in finite time \(T\)  on the sphere of raius \(r=r_{max}\);
	\item there holds that for all \(R > 0\),
	\[
	\sup_{\Lambda_R} \left|(T - t)^{\frac{1}{p-1}} u(r, t) - f\right| \to 0 \text{ as } t \to T, \quad (7)
	\]
	where \(\Lambda_R := \left\{ |r-r_{max}| \leq R \sqrt{(T - t)|\log(T - t)|}\right\}\) and where the function \(f=(p-1+b\frac{y^2}{s})\) and \(b=\frac{(p-1)^2}{4p}\) 
\end{enumerate}
\end{theorem}
\end{comment}
\begin{theorem}
\label{th: Main theorem}
(Existence of a singular standing sphere solution for equation \eqref{nlh-introduction} with prescribed profile).
There exists \(T > 0\) such that equation \eqref{nlh-introduction} has a solution \(u(x, t)\) in \(\Rbb^d \times [0, T)\), with radial symmetry such that:
\begin{enumerate}
\item The solution \(u\) blows up in finite time \(T\) on the sphere of radius \(r_{max}\);
\item There holds that for all \(R > 0\),
\[
\sup_{\Lambda_R} \left|(T - t)^{\frac{1}{p-1}} u(|x|, t) - f\left (\frac{|x|-r_{max}}{\sqrt{(T-t)|\log(T-t)|}}\right )\right| \to 0 \text{ as } t \to T, 
\]
where \(\Lambda_R := \left\{ ||x|-r_{max}| \leq R \sqrt{(T - t)|\log(T - t)|}\right\},\)
\[f(z)=\left (p-1+bz^2\right )^{-\frac{1}{p-1}}\mbox{and }b=\frac{(p-1)^2}{4p}.\]
\item for all $r>0$, $r\neq r_{max}$, $u(r,t)\rightarrow u(r,T)$  as $t\rightarrow T$, with $u(r,T)\sim u^*(r-r_{max})$ as $r\rightarrow r_{max}$, where
\begin{equation}\label{eq: profile}
	u^*(r)\sim \left[\frac b2\frac{r^2}{|\log r|}\right]^{-\frac{1}{p-1}}\mbox{ as }r\to 0.
	%\left[\frac{(p-1)^2r^2}{8p|\log r|}\right]^{-\frac{1}{p-1}}
\end{equation}
\end{enumerate}
\end{theorem}

Using  ideas from \cite{MZdm97} and \cite{ZAAihn98}, we are able to interpret the two-dimensional variable in terms of blow-up point and blow-up time. This leads to the stability of the profile \eqref{eq: profile} in Theorem \ref{th: Main theorem}.
\begin{proposition}[Stability of the singular standing sphere solution] \label{Prop: stability}

Denote by $\hat{u}$ the solution constructed in Theorem \ref{th: Main theorem} that blows up at the sphere of radius $\hat{r}$ and note by $T_{\hat{u}}$ its blow-up time. Then, there exists $\varepsilon_0 > 0$ such that for any radially symmetric initial data $u_0 \in H$, satisfying $\|u_0 - \hat{u}(\cdot, 0)\|_{L^\infty} \leq \varepsilon_0$, the solution of \eqref{nlh-introduction}, with initial data $u_0$, blows up at finite time $T_{u_0}$ at only one collapsing ring with radius  $r_0$  in $\Rbb^n$ . Moreover, the function $u(|x|, \cdot)$ satisfies the same estimates as $u$ with $T_{\hat{u}}$ replaced by $T_{u_0}$. Furthermore, it follows that
\[ T_{u_0}\to T_{\hat{u}}, \quad r_{0} \to \hat{r} \quad \text{as} \quad u_0 \to \hat{u}(0). \]

\end{proposition}
%	\begin{proof}
%		\end{proof}
%\textcolor{red}{Here, it would be nice to say in a remark that the solution is unstable under non-radial perturbations, thanks to the (unpublished) result about genericity by Herrero and Vel\'azquez....}

\begin{remark}
To prove Theorem \ref{th: Main theorem}, we project the linearized partial differential equation on the eigenfunctions $h_m$ given by \eqref{eq:definition of hm}.  This is technically different from the work of \cite{MNZNon2016}, \cite{MZdm97}, and \cite{BKLcpam94}, where the authors use the integral equation. We will follow the two steps proposed in \cite{MZdm97} but in a more straight way.  Indeed, we have an additional problem coming from the fact that the equation in radial coordinates presents a singularity in zero. To solve this problem, we will use ideas from \cite{MNZNon2016} and \cite{Rdm06}. 
\end{remark}
\begin{remark}
In this paper, we are focused on the radial dynamics of the circle that reduces to the one-dimensional dynamic. We will give the proof in dimension 2, but it can be extended to a higher dimension with no difficulties.
\end{remark}
\begin{remark}Note that Herrero and Vel\'azquez showed the genericity of the behavior given by \eqref{eq: profile generic} in \cite{HVCRAS92} and \cite{HVasnsp92} dedicated to the one dimensional case, and in a non-published document in higher space dimensions. In Proposition \ref{Prop: stability}, we focused on the radially symmetric perturbations. However, under non-radial perturbations, due to the genericity of the profile,  the stability of the blow-up profile breaks down.
%\eqref{eq: profile generic} shown in \cite{HVunpublished-92}}
\end{remark}
This paper is organized as follows.
In Section \ref{Section-formulation-of-the-problem}, we will give the formulation of our problem. Then, in Section \ref{Section-proof-without-technical-details}, we give the proof of Theorem \ref{th: Main theorem} without technical details and solve the finite dimension problem. Finally, in Section \ref{Section-reduction-to-finite-dimension}, we conclude by giving the proofs of propositions cited in Section  \ref{Section-proof-without-technical-details}.

\section{Formulation of the problem}\label{Section-formulation-of-the-problem}
For simplicity, we give the proof in dimension $d=2$. Inspired by the numerical results \cite{BFGpd10}, we consider the radially symmetric solution $u(r,t)=u(|x|,t)$, then we write equation \eqref{nlh-introduction} in radial coordinates as follows

\begin{equation}\label{equation-NLH-1d}
\del_t u=\del^2_{r} u+\frac{d-1}{r}\del_r u+|u|^{p-1}u.
\end{equation}

In general, the two terms $\partial_{r}^2u$ and $\frac{\partial_r u}{r}$ forming the Laplacian certainly scale the same way. Heuristically, if we assume that the singularity formation of a priory takes place exclusively around the circle $r\sim 1$, then  on this circle, the term $ \frac {\partial _r u}{r}$ scales below $\partial _r^2 u$, and thus the singular part of the equation is governed by the one-dimension nonlinear heat equation
\begin{equation}
\partial_t u= \partial_{r}^2u+|u|^{p-1}u,\;
\end{equation}
for which the rigorous construction of blow-up solution is very well -known (see \cite{MZdm97}),  while the existence of the term $\frac{d-1}{r}\del_r u$, which has a singularity at $\{(t,x)| x=0\}$, prevents us from the estimations in a neighborhood of the origin.
% However, one can easily show that the new equation is obtained by performing a change of variable on the one dealt with in \cite{MZdm97}.
We naturally think of separating the space into two parts: the regular part and the blow-up part. The first part contains the origin, where the solution is supposed to be regular, while the other part is away from the origin and the solution is expected to be explosive. 

\medskip
We  introduce the following smooth nonnegative cut-off functions:
\begin{equation}\label{eq: def Chi}
\chi= \left\{
\begin{array}{rcl}
0 &&  0\leq\xi \leq \frac{1}{8} ,\\
&& \\       
1 &&  \xi\geq \frac{1}{4},\\
\end{array}\right.\end{equation}
and
\begin{equation}
\overline{\chi}= \left\{
\begin{array}{rcl}
0 && \xi\geq \frac{3}{4},\\
&&\\
1 &&  0 \leq\xi \leq\frac{3}{8}.\\
\end{array}\right.
\end{equation}
%Let $\overline{u}(x,t)=\overline{\Xcal}\left(\frac{|x|}{\vep_0}\right)u(x,t)$, where $u(x,t)$ is assumed to satisfy the following equation:
%$$u_t = \Delta u+|u|^{p-1}u.$$ 
\begin{comment}
as well as $U(r,t)=u(x,t)$ with $r=|x|$ in the region $\Omega\coloneqq \{x| x\in \Rbb^d,\ |x|\geq \frac{1}{2}\}$. we will discuss their properties separately:
\end{comment}
\textbf{In the regular region},  we define $\overline{u}(x,t)=\overline{\chi}\left(\frac{|x|}{\vep_0}\right)u(x,t)$, for $x\in\mathbb{R}^2$, where $u(x,t)$ is assumed to satisfy the following:
$$u_t = \Delta u+|u|^{p-1}u.$$ 
Then, for all $x\in \mathbb{R}$, $\overline{u} $ satisfies the following equation:
\begin{equation}
\del_t \overline{u} = \Delta \overline{u} + |u|^{p-1}\overline{u}-2\nabla \overline{\chi}\nabla u - \Delta \overline{\chi} u 
\end{equation}
$\overline{u}$ will be controlled using classical parabolic estimates.
\medskip

\noindent\textbf{In the blow-up region}:  First, we note that by an invariable scaling, we can take $r_{max}=1$. In the following, we consider the equation in radial coordinates given by \eqref{equation-NLH-1d}.

Let us introduce $U(r,t)=u(|x|,t)$ with $r=|x|$, then $U$ satisfies the following equation:
\begin{equation}\label{eq: U }
\del_t U=\del^2_{r} U+\frac{d-1}{r}\del_r U+|U|^{p-1}U.
\end{equation}
If we consider the following self-similar transformation: 
\begin{equation}\label{self-similar-variables}
%\begin{aligned}
%y&=\frac{r-1}{\sqrt{T-t}},\\
%s&=-\log(T-t),\\
W(y,s)=(T-t)^{\frac{1}{p-1}}U(r,t)\mbox{ with }y=\frac{r-1}{\sqrt{T-t}},\; s=-\log(T-t),
%\end{aligned}
\end{equation}
then $W$  satisfies:
\begin{equation}\label{eq: W}
\del_s W =\del^2_{y}W-\frac{1}{2}y\del_{y}W+e^{-s/2}\frac{d-1}{ye^{-s/2}+1}\del_y W-\frac{W(y,s)}{p-1}+|W|^{p-1} W,
\end{equation}
with $y\in[-e^{s/2},+\infty)$ and $s\in [-\log T,+\infty)$. If we set $w = W\cdot \chi(\frac{ye^{-s/2} + 1}{\vep_0})$, where $\Xcal$ is defined by \eqref{eq: def Chi},
\begin{comment}
Multiplying $\Xcal$ on both sides of the equation, we obtain: 
$$
\Xcal\del_s W= \Xcal\del^2_y W-\frac{W(y,s)}{p-1}-\Xcal\frac{1}{2}y\del_{y}W+\Xcal|W|^{p-1} W+e^{-s/2}\Xcal\frac{d-1}{ye^{-s/2}+1}\del_y W.
$$
\end{comment}
then $w$ satisfies:  
\begin{equation}
\del_s w = \del^2_y w-\frac{1}{2}y\del_y w-\frac{1}{p-1}w - |w|^{p-1}w + e^{-s/2}\frac{d-1}{ye^{-s/2}+1}\del_y w + F(y,s),
\end{equation}
where $F(y,s) $ is defined as follows.

\begin{equation}	
F(y,s) =\left\{\begin{aligned}
&\begin{split}
&W\del_s \chi - 2\del_y \chi\del_y W- W\del_y^2\chi\\
&+\frac{1}{2}yW\del_y\chi- \frac{d-1}{y+e^{s/2}}W\del_y\chi+ |W|^{p-1}W(\chi-\chi^{p})
\end{split},&\text{ if }y \geq -\frac{3}{4}e^{-s/2}\\
\bigskip
& 0, &\text{otherwise.}
\end{aligned}\right.
\end{equation}

In the ring $\{r=1\}$, we introduce the perturbation $q$ defined by
$$
w=\varphi+q,
$$
with
\begin{equation} \label{definition of varphi}
\varphi= f\left(\frac{y}{\sqrt{s}}\right)+\frac{\kappa}{2ps},
\end{equation}
where 
\begin{equation} \label{definition of f}
f(z)=(p-1+bz^2)^{-\frac{1}{p-1}},\ \kappa=(p-1)^{-\frac{1}{p-1}}, \text{ and } b=\frac{(p-1)^2}{4p}.
\end{equation}

The problem is then reduced to constructing a function $q$ satisfying 
$$\lim_{s\to \infty}\sup_{y\in [-e^{-\frac{s}{2}},+\infty)}|q(y,s)|=0.$$
The equation for $q$ is as follows:
\begin{equation}\label{q's equation}
\del_s q = (\mathcal{L}+V)q+H(y,s)+\del_y G(y,s)+R(y,s)+B(y,s)+N(y,s)
\end{equation}
where
\begin{equation}
\mathcal{L}=\del^2_y-\frac{1}{2}y\del_y+1,\quad V= p\varphi^{p-1}-\frac{p}{p-1},
\end{equation}
\begin{equation}
B(y,s)=|\varphi+q|^{p-1}|\varphi+q|-\varphi^p-p\varphi^{p-1}q,
\end{equation}
and
\begin{equation}\label{definition of B,R,F,N}
\begin{aligned}
R(y,s)&= \del^2_y \varphi-\frac{1}{2}y\del_y \varphi-\frac{1}{p-1}\varphi+\varphi^p-\del_s \varphi,\\
H(y,s)&=W(\del_y^2 \chi+\del_s \chi+\frac{1}{2}y\del_y\chi\del_y\chi)+ |W|^{p-1}W(\chi-\chi^{p}),\\
G(y,s)&= -2\del_y\chi W,\\
N(y,s)&= \frac{d-1}{y+e^{s/2}}W\del_y\chi.
\end{aligned}
\end{equation}
The control of $q$ near the collapsing ring $\{r=1\}$ obeys two facts:
\medskip

\noindent\textbullet \textbf{ Localization}\\
Looking at the expression provided in \eqref{definition of varphi}, we note that the variable $z=\frac{y}{\sqrt{s}}$ plays a fundamental role. Consequently, we will analyze the behavior of $q$ separately when $|z|>2K$, namely the outer region, and when $|z|\leq2K$, namely the inner region, with the sufficiently large specific value of $K>0$ to be chosen later.

%Thus, we decompose $q$ as follows:
Let us consider the cut-off function  $\chi_0\in C_0^{\infty}([0,+\infty))$, %with support $[0,2]$ and 
%$\Xcal_0\equiv 1$ on $[0,1]$
such that $\Xcal_0(\xi)= 1$
for $\xi < 1$ and $\chi_0(\xi)= 0$
for $\xi > 2$ and introduce

%. Then, we define
\begin{equation}\label{def_chi_c}
\chi_c (y,s)=\chi_0\left (\frac{|y|}{2K \sqrt{s}}\right )\mbox{ where $K >0$ is chosen large enough,}
\end{equation}
and we introduce:
\begin{equation}
q_e=q(1-\chi_c).
\end{equation}
\noindent\textbullet \textbf{ Spectral properties of the linear operator $\mathcal{L}$}\\

\medskip

The operator $\mathcal{L}$ is self-adjoint on $\mathcal{D}(\mathcal{L})\subset L^{2}(\mathbb{R},d\mu)$ with 
$$d\mu(y)=\frac{e^{-\frac{y^2}{4}}}{(4\pi)^{1/2}}dy.$$ The spectrum of $\mathcal{L}$ is 
$$spec(\mathcal{L})=\{1-\frac{m}{2}|m\in \mathbb{N}\}.$$
All the eigenvalues are simple and the corresponding eigenfunctions are derived from Hermite polynomials:
\begin{equation}\label{eq:definition of hm}
h_m(y)=\sum_{n=0}^{[\frac{m}{2}]} \frac{m!}{n!(m-2n)!}(-1)^ny^{m-2n}. 
\end{equation}
$h_m$ satisfies 
$$\int_\mathbb{R} h_mh_nd\mu=2^nn!\delta_{nm}.$$

Thanks to the above spectral properties, we can define the following projections: \\

\noindent\textbullet\textbf{Decomposition of q}

\medskip

For the sake of controlling $q$ in the region $|y|<\sqrt{s}$, we will expand the unknown function $q$ (and not $q\chi_c$) concerning the Hermite polynomial.

%\textbullet\  For $m\in\{0,1,2\}$
\begin{equation}\label{eq: Def of projections P_m}
P_m (f)=f_m=\displaystyle \frac{ \displaystyle\int_{\Rbb} f h_m d\mu}{\displaystyle\left( \int h_m^2 d\mu\right )^{1/2}}\mbox{ for $m\in\{0,1,2\}$,}
\end{equation}

\begin{equation}\label{eq: Def of projection P_-}
P_-(f)=f_-=\sum_{m\geq 3}P_m (f).
\end{equation}
Then we study 
\begin{equation}\label{decomposition of q}
q(y,s)=\sum_{m=0}^2 q_m(s) h_m(y)+q_{-}(y,s).
\end{equation}
%and hence 

%$$
%q(y,s)=\sum_{m=0}^2 q_m(s) h_m(y)+q_{-}(y,s)+q_e(y,s).
%$$

\section{The existence assuming some technical results}\label{Section-proof-without-technical-details}

This section is devoted to the proof of Theorem \ref{th: Main theorem} and as, mentioned before, we only give the proof in $\Rbb^2$. We proceed in four steps, each of them making a separate subsection.
\begin{itemize}
\item In the first subsection, we define the bootstrap regime and translate our goal of making $q(s)$ go to 0 in terms of belonging to $\mathcal{S}$.
\item In the second subsection, we give an initial data family for equation \eqref{nlh-introduction}, such that the initial datum is trapped in the shrinking set.
\item In the third subsection, using spectral properties of the linearized operator in the blow-up region and parabolic regularity in the regular region, we reduce our goal from the control of $u\in \mathcal{S}$ to the control of the two first components of $q$ ($q_0$ and $q_1$).
\item  We end this section by solving the finite-dimensional problem using the shooting lemma and conclude the proof of Theorem \ref{th: Main theorem}.
\end{itemize}

\subsection{Bootstrap regime}
In this part, we introduce the following shrinking set
\begin{definition}
\label{Def: shrinking set}
%For all $K_0 > 0$, $\varepsilon_0 > 0$, $A > 0$, $0 < \eta_0 \leq 1$, and $T > 0$, we define for all $t \in [0, T)$ the set $V(K_0, \varepsilon_0, A, \eta_0, T, t)$ as being the set of all functions $u \in L^\infty(\mathbb{R}^2)$ satisfying:
For $A$, $ K_0$, $\varepsilon_0>0$, $0<\eta_0\leq 1$, $T>0$, we define for all $t \in [0, T)$   $$\mathcal{S}(t)=\mathcal{S}[A,K_0, \varepsilon_0,\eta_0](t),$$
the set of all functions $u \in L^\infty(\mathbb{R}^2)$ satisfying:
\begin{itemize}
\item 	(i) Estimates in $\mathcal{R}_1$: We consider $ \mathcal{V}(s)= \mathcal{V}[K_0,A](s)$ (where $s=-\log (T-t)$ ), the set of all functions $r \in L^\infty(\mathbb{R})$ such that
\[
\begin{aligned}
\lvert r_m(s) \rvert &\leq A s^{-2} \quad (m = 0, 1), \\
\lvert r_2(s) \rvert &\leq A^2 s^{-2} \log(s), \\
\lvert r_- (y, s) \rvert &\leq A s^{-2} (1 + \lvert y \rvert^3), \\
\lvert r_e(y, s) \rvert &\leq A s^{-1/2},
\end{aligned}
\]
where
\[
\begin{aligned}
r_e(y, s) &= (1 - \chi_c(y, s)) r(y, s), \\
r_-(s) &= P_-( r),
\end{aligned}
\]
for $m \in \mathbb{N}$, where $r_m(y, s) $ and $P_-$ are defined in \eqref{eq: Def of projections P_m} and \eqref{eq: Def of projection P_-}.

\item 	(ii) Estimates in $\mathcal{R}_2$: For all $0 \leq \lvert x \rvert \leq \frac{3\vep_0}{4}$, $\lvert u(x, t) \rvert \leq \eta_0$.
\end{itemize}
\end{definition}
This definition yields the following a priori estimates on the functions in $\mathcal{V}(s)$.
\begin{proposition}\label{Prop: Priori estimate}
For any $s>1$, let $r$ be in the shrinking set $\mathcal{V}(s)$ defined in Definition \ref{Def: shrinking set}. Then, the following estimates hold.
\begin{enumerate}
\item $\|r\|_{L_{\infty}(\Rbb)} \leq C(K)\frac{A^2}{\sqrt{s}}$,
\item for all  $ y\in \Rbb $, $|r(y)|\leq CA\frac{\log s}{s^2}(1+|y|^3)$
\end{enumerate}
\end{proposition}
\begin{proof}
The proof is the same as Proposition 4.1 in \cite{MZjfa08}; hence, we omit it here.
\end{proof}
\subsection{Preparation of initial data}
In this part, we aim to give a suitable family of initial data for our problem. Let us consider $(d_0,d_1 )\in \mathbb{R}^2$, $T>0$; we consider the initial data for the equation \eqref{nlh-introduction}
defined for all $x\in \Rbb^{2}$ by

\begin{equation}\label{eq: initial data}
u_0(x,d_0,d_1)= T^{-\frac{1}{p-1}}\left\{ \varphi(y,s_0)\chi(\frac{ye^{-s_0/2}}{\vep_0})+\frac{A}{s_0^2}(d_0+d_1 y)\chi_c\right\},
\end{equation}
where $s_0=-log T$, $y=\frac{|x|-1}{\sqrt{T}}$, $\chi$ is defined in \eqref{eq: def Chi}, and $\chi_c$ is given by \eqref{def_chi_c}.

%In the following lemma we show the existence of a set $D_{K_0, \varepsilon_0, A, T} = D_T$ for $(d_0,d_1)\in D_t$ the initial data is trapped in the shrinking set: 

%\begin{lemma} 
%	There exists $K_{02} > 0$ such that for each $K_0 \geq K_{02}$ there exist $\varepsilon_0 > 0$, $A \geq 1$, there exists $s_{0,1}(K_0, \varepsilon_0, A) \geq 0$ such that for all $s \geq s_{0,1}$: If initial data for \eqref{nlh-introduction} are given by \eqref{eq: initial data} then, there exists a rectangle 
%	\begin{equation}\label{Def: rectangle for initial data}
%	D_{K_0, \varepsilon_0, A, T} = D_T \subset [-2, 2]^2
%\end{equation}
%such that, for all $(d_0, d_1) \in D_T$, we have $u(K_0, T, A, d_0, d_1) \in V(K_0,A,s)$.
%\end{lemma}

%, we hence claim the following similar lemma: 
\begin{lemma}\label{Lem: preparation Inditial data}[ decomposition of initial data in different components ]
There exists $K_{0} > 0$ such that for $\varepsilon_0 > 0$, $A \geq 1$, there exists $s_{0}(K_0, \varepsilon_0, A) \geq e$ such that :
\begin{enumerate}
\item There exists a rectangle 
\begin{equation}\label{def: Dt}
	\mathcal	D_{K_0, \varepsilon_0, A, T} =\mathcal D_T \subset [-2, 2]^2,
\end{equation}
such that the mapping $(d_0, d_1) \to (q_0(s_0), q_1(s_0))$ is linear and one-to-one from $\mathcal D_T$ onto $[-\frac{A}{s_{0}^2}, \frac{A}{s_{0}^2}]^2$ and maps the boundary $\partial\mathcal D_T$ into the boundary $\partial [-\frac{A}{s_{0}^2}, \frac{A}{s_{0}^2}]^2$. Moreover, it is of degree one on the boundary.
\item For all $(d_0, d_1) \in\mathcal D_T$, we have: 
\begin{equation}
	\begin{aligned}
		&|q_2(s_0)|  \leq CAe^{-s_0},\ \  |q_-(y, s_0)| \leq \frac{c} {s_0^2}(1 + |y|^3),\\ 
		&\text{ and } q_e(y, s_0) = 0,\ \  |d_0| + |d_1| \leq 1.
	\end{aligned}
\end{equation}

\item For all $(d_0, d_1) \in\mathcal D_T$ and $|x| \leq \varepsilon_0/4 $, we have $u(x, d_0, d_1) = 0$.
\end{enumerate}
\end{lemma}

\begin{proof}The proof is purely technical and follows as the analogous step in  \cite{MNZNon2016} and \cite{MZdm97}; for that reason we refer the reader to Lemma 3.5, page 156 and Lemma 3.9, page 160 in \cite{MZdm97}.
\end{proof}
%Since we have almost the same expression of the initial data as in paragraph 3.2 page 12 in \cite{MNZNon2016}
%Using Lemma \ref{Lem: preparation Inditial data}}, from the definition of \eqref{eq: initial data},  we can easily conclude that $u(s_0)\in \mathcal{S}(T)$.\\

%Then by the local well-possednes of the solution to \eqref{nlh-introduction}

%In the next subsection, we assume that there exists $s_0< s_1< +\infty$, such that $s_1=\sup\{s| u(t) \in \mathcal{S}(t)\}$.  Our goal will be to reduce the dimension of the problem. In fact, we find when it cames to $s_1$, Proposition \ref{Prop:control of q by (q1,q2)} shows that $q(s_1)$ can quit $\mathcal \mathcal{S}(T)$ only by the unstable modes given by. $(q_0,q_1)$.  Finally, for this finite dimentional problem, we close the bootstrap argument by a contradiction induced with topological degree theorem.

\subsection{Reduction to a finite-dimensional problem}

In this part, we show that the control of the infinite problem is reduced to a finite-dimensional one.  Since the definition of  the bootstrap $\mathcal S(s)$ shows two different types of
estimates, in the regions $\mathcal R_1$ and $\mathcal{R}_2$,  we need two different approaches to
handle those estimates:
\begin{itemize}
\item In  $\mathcal{R}_1$, we work in similarity variables \eqref{self-similar-variables}; in particular, we crucially use the projection of equation \eqref{q's equation} with respect to the decomposition given in \eqref{decomposition of q}.
\item  In,  $\mathcal{R}_2$, we directly work in the variables $u(x, t)$, using standard parabolic estimates. For more details, see subsection \ref{section_regular_region}.
\end{itemize}

%It follows from the \eqref{eq:estimation on u in regularregion} that the solution will stay in regular region if the initial data is in the regular region, hence in this 

In the following, we restrict ourselves to the blow-up region. It is sufficient to prove there exists a unique global solution $q$ on $[s_0,+\infty)$ for some $s_0$ large enough such that
\[q(s) \in \mathcal{V}(s),\; \forall s\geq s_0.\]
In particular, we show that the control of the infinite problem is reduced to a finite-dimensional one. To obtain this key result, we first claim the following a priori estimates. We should emphasize that the parameters $K$, $A$,$T$ and $s_0$ in the following lemmas are allowed to vary from one to one. When proving Proposition \ref{Prop:control of q1,q2,q-,qe}, we will prove that the conclusions of all lemmas are simultaneously valid for values
of $K$, $A$, $T$, and $s_0$ as described in the proposition.

\begin{proposition}[A prior estimates]\label{Prop:control of q1,q2,q-,qe}
There exists $A \geq 1$ and $s_0\geq 0$ such that for all $s\geq s_0$ if $q(s)\in \mathcal{V}(s)$ is true, then the following holds: 
\begin{enumerate}
\item(Ordianary differential equation satisfied by the expanding models) For $m=0$, or $1$, we have 
\begin{equation}
\left| q'_m-(1-\frac{m}{2})q_m\right|\leq \frac{AC}{s^{2}}.
\end{equation}
\item(Control of null and negative modes) 
\begin{equation*}
\begin{aligned}
	\left|q_2(s)\right|&\leq \left(\frac{\tau}{s}\right)^2q_2(\tau) + CA^2s^{-2}\log(s/\tau)\\ 
	\left\|\frac{q_-(s)}{1+|y|^3}\right\|_{L^{\infty}}&\leq e^{-\frac{3}{4}(s-\tau)}\left\|\frac{q_-(\tau)}{1+|y|^3}\right\|_{L^{\infty}}+\frac{CA^2}{s^2}\end{aligned} 
	\end{equation*} 
	\item  (Control of outer part $q_e$) \begin{equation*}
\begin{aligned}
	\|q_e(s)\|_{L^\infty}&\leq e^{-\frac{(s-\tau)}{2(p-1)}}
	\|q_e(\tau )\|_{L^\infty}+C\frac{A^2}{\sqrt{\tau}}(1+s-\tau). 	
\end{aligned} 
\end{equation*} 
\end{enumerate}
\end{proposition}
The idea of the proof of Proposition \ref{Prop:control of q1,q2,q-,qe} is to project \eqref{q's equation} according to the decomposition \eqref{decomposition of q}. 
The computations are too long, so we postpone the proof of Proposition \ref{Prop:control of q1,q2,q-,qe} to the whole section \ref{Proof-a-priori-estimates}.    

\medskip

Consequently, we have the following result

\begin{proposition}[Control of $q(s)$ in $\mathcal V(s)$ by $(q_0(s), q_1(s))$]\label{Prop:control of q by (q1,q2)}
There exists $A > 1$ such that there exists $T(A) \in (0, 1/e)$ such that the following holds:
If $q$ is a solution of $(21)$--$(51)$ with initial data at $s = s_0 = -\log T$ given by $(45)$ with $(d_0, d_1) \in D_T$, and $q(s) \in \mathcal S(s)$ for all $s \in [s_0, s_1]$ with $q(s_1) \in \partial \mathcal S(s_1)$ for some $s_1 > s_0$, then:
\begin{enumerate}
\item[(i)] $(q_0(s_1), q_1(s_1)) \in \partial[-\frac{A}{s_1^2}, \frac{A}{s_1^2}]^2$.
\item[(ii)] (Transverse crossing) There exists $m \in \{0, 1\}$ and $\omega \in \{-1, 1\}$ such that 
\[\omega q_m(s_1) =\frac{A}{s_1^2} \mbox{ and }\omega \frac{d}{ds} q_m(s_1) > 0.\]
\end{enumerate}
\end{proposition}

\begin{remark}
In (ii) of Proposition \ref{Prop:control of q by (q1,q2)}, we show that the solution $q(s)$ crosses the boundary $\partial \mathcal{V}(s)$ at $s_1$  with positive speed; in other words, all points on $\partial \mathcal{V}(s_1)$ are strict exit points. The construction is essentially an adaptation of Wazewski's principle (see \cite{Conbook78}, chapter II and the references given there).

\end{remark}

\begin{proof}[Proof of Proposition \ref{Prop:control of q by (q1,q2)}]
Assuming Proposition \ref{Prop:control of q1,q2,q-,qe}, we argue as in the proof of Proposition 4.5, page 1632 from \cite{MZjfa08}.
By choosing proper $A$ and $T$, we can use the conclusions of Proposition 4.6.

To prove (i), we notice that from Definition \ref{Def: shrinking set} and the fact that $q_0(s) = 0$, it is enough to show that for all $s \in [s_0, s_1]$,
\begin{equation}\label{open condition bootstrap}
\begin{aligned}
&\|q_e\|_{L^\infty(\mathbb{R})} \leq \frac{A}{2\sqrt{s}} ,\\
&\|q_- (y)\|_{L^\infty(\mathbb{R})} \leq \frac{A(1 + |y|)^{3}}{2s^2} ,\\
&|q_2| \leq \frac{A^2}{2s^2}.
\end{aligned}
\end{equation}

Define $\sigma = \log A$ and take $s_0 \geq \sigma$ (that is, $T = e^{-\sigma} = 1/A$) so that for all $\tau \geq s_0$ and $s \in [\tau,\tau + \sigma]$, we have
\[
\tau \leq s \leq \tau + \sigma \leq \tau + s_0 \leq 2\tau \implies \frac{1}{2} \leq \frac{\tau}{s} \leq \frac{s}{\tau}.
\]

We consider two cases in the proof.

\textbf{Case 1: $s \leq s_0 + \sigma$}. Note that (54) holds with $\tau = s_0$. Using (ii) of Proposition \ref{Prop:control of q1,q2,q-,qe} and estimate (ii) of Proposition 4.2 on the initial data $q(\cdot, s_0)$, we write
\[
\begin{aligned}
& \|q_2(s)\| \leq CA^2e^{-\gamma s/2} + \frac{CA^2}{s^2},\\
& \|\frac{q_-(s)}{1+|y|^3}\|_{L^\infty} \leq C\frac{A}{(s/2)^3}+C\frac{A^2}{s^2},\\
& \|q_e(s)\|_{L^\infty} \leq CA^{3} (s/2)^{-1/2}(1 + \log A).
\end{aligned}
\]
Thus, for sufficiently large $A$ and $s_0$, we see that \eqref{open condition bootstrap} holds.

\textbf{Case 2: $s > s_0 + \sigma$}.  Let $\tau=s-\sigma >s_0$, by Proposition \ref{Prop:control of q1,q2,q-,qe}   and using the fact that $q(\tau)\in\mathcal V(\tau) $, we write 

\[
\begin{aligned}
& \|q_2(s)\| \leq A^2(s/2)^{-2}\log(s)  + \frac{CA^2}{2}s^{-2}\log(s),\\
& \left\|\frac{q_-(s)}{1+|y|^3}\right\|_{L^{\infty}} \leq e^{-\frac{3}{4}\sigma}\frac{A}{(s/2)^2}+C\frac{A^2}{s^2},\\
& \|q_e(s)\|_{L^\infty} \leq e^{-\frac{\sigma}{2(p-1)}}\frac{A}{(s/2)^{1/2}}+\frac{CA^2}{(s/2)^{1/2}}(1+\sigma).
\end{aligned}
\]
Thus, in this case, we see clearly that there exists sufficiently large $A$ and $s_0$ such that conditions in \eqref{open condition bootstrap}  are satisfied.

Conclusion of (i): We select $A$, and $s_0$ large enough so that \eqref{open condition bootstrap} are verified. Then, the fact that $q(s_1) \in \del \mathcal V(s_1)$ together with the definition of $\mathcal V(s)$ shows that (i) of \ref{Prop:control of q by (q1,q2)} is true. 
From (i) in \ref{Prop:control of q by (q1,q2)}, we deduce (ii) as follows:

From (i), there is $(m,\omega)\in \{0,1\}\times \{-1,1\}$ such that $q_m(s_1)=\omega\frac{A}{s_1^2}$ , and using 1 of \ref{Prop:control of q1,q2,q-,qe}, we see that 
\begin{equation}
\omega q'_m(s_1)\geq \left(1-\frac{m}{2}\right)\omega q_m(s_1)-\frac{C}{s^2_1}\geq \frac{(1-m/2)A-C}{s_1^2}.
\end{equation}
Taking $A$ large enough concludes the proof of Proposition \ref{Prop:control of q by (q1,q2)}.

\end{proof}

\subsection{Control of the solution in the bootstrap regime and proof of Theorem \ref{th: Main theorem}}

We prove Theorem \ref{th: Main theorem} using the previous results. We proceed in two parts:\\
\textbf{Part 1: Solution to the finite-dimensional problem}

Let $A$, and $T(=e^{-s_0})$ be chosen so that Proposition \ref{Prop:control of q by (q1,q2)} and Proposition \ref{Prop:control of q1,q2,q-,qe} are valid, We will find the parameters $(d_0,d_1)\in \mathcal{D}_T$ defined in \eqref{def: Dt} and advance by assuming that for all $(d_0,d_1)\in \mathcal{D}_T$, there exists $s_*(d_0,d_1) \geq -\log T$ such that $q_{d_0,d_1}(s) \in\mathcal V(s)$ for all $s \in [-\log T, s_*]$ and $q_{d_0,d_1}(s_*) \in \partial\mathcal V(s_*)$. From (i) of Proposition \ref{Prop:control of q by (q1,q2)}, we see that $(\tilde{q}_0(s_*), \tilde{q}_1(s_*)) \in \partial[-\frac{A}{s_*^2}, \frac{A}{s_*^2}]^2$ and the following function is well-defined:

\begin{equation}\label{definition of Phi}
\begin{aligned}
\Phi: \mathcal{D}_T&\rightarrow\partial[-1, 1]\\
(d_0,d_1)&\rightarrow \frac{s_*^2}{A}(\tilde{q}_0, \tilde{q}_1)_{d_0,d_1}(s_*).
\end{aligned}
\end{equation}
This function is continuous by (ii) of Proposition \ref{Prop:control of q by (q1,q2)}. If we manage to show that $\Phi$ is of degree 1 on the boundary, then we have a contradiction from the degree theory. We now focus on proving that.

Using the fact that $q(-\log T) = \psi_{d_0,d_1}$, we see that when $(d_0,d_1)$ is on the boundary of the quadrilateral $\mathcal{D}_T$, $(\tilde{q}_0, \tilde{q}_1)(-\log T) \in \partial[-A(\log T)^{-2}, A(\log T)^{-2}]^2$ and $q(-\log T) \in \mathcal V_A(-\log T)$ with strict inequalities for the other components. Applying the transverse crossing property of Proposition \ref{Prop:control of q by (q1,q2)}, we see that $q(s)$ leaves $\mathcal V(s)$ at $s = -\log T$, hence $s_*(d_0,d_1) = -\log T$. Using \eqref{definition of Phi}, we see that the restriction of $\Phi$ to the boundary is of degree 1. A contradiction then follows. Thus, there exists a value $(d_0,d_1) \in \mathcal{D}_T$ such that for all $s \geq -\log T$, $q_{d_0,d_1}(s) \in \mathcal V (s)$.\\ 
\textbf{Part 2: Proof of Theorem 1}

Consider the solution constructed in Part 1, such that $q(s)\in \mathcal V(s)$.
Then by Definition \ref{Def: shrinking set}, we see that
\[
\forall y \in \mathbb{R}, \quad \forall s \geq -\log T, \quad |q(y, s)| \leq \frac{CA^2 }{\sqrt{s}}.
\]

By definitions \eqref{self-similar-variables} \eqref{definition of varphi}, we see that
\[
\forall s \geq -\log T, \quad \forall |x| \geq \frac{ \varepsilon_0}{4}, \quad \left|W(y, s) - f\left(\frac{y}{\sqrt{s}}\right) \right|\leq \frac{CA^2}{\sqrt{s}} + \frac{C}{s}.
\]

By definition \eqref{self-similar-variables} of $W$, we see that $\forall t \in [0, T), \quad \forall |x| \geq \frac{\varepsilon_0}{4},$
\[
\left|(T - t)^{1/(p-1)} u(r, t) - f\left(\frac{r-r_{max}}{ \sqrt{(T-t)\log(T - t)}}\right)\right| \leq \frac{C(A)}{\sqrt{|\log(T - t)|}}.
\]

(i) If $r = r_{max}(=1)$, then we see from above that $|u(0, t)| \sim \kappa(T - t)^{-1/(p-1)}$ as $t \rightarrow T$. Hence, $u$ blows up at time $T$ at $r = r_{max}$.

It remains to prove that any $r\neq r_{max}(=1)$ is not a blow-up point.  Since we know from item (ii) in Definition \ref{Def: shrinking set} that if $r \leq \frac{3\vep_0}{4}$, and $0 \leq t \leq T$, $|u(r, t)| \leq \eta_0$, it follows that $r$ is not a blow-up point, provided $r \leq \frac{3\vep_0}{4}$.

Now, if $r \geq \frac{3\vep_0}{4}$, the following result from Giga and Kohn [13] allows us to conclude.

\begin{proposition}[Giga and Kohn]\label{threshold of non-blow-up}

For all $C_0 > 0$, there is $\eta_0 > 0$ such that if $v(\xi, \tau)$ solves
\[
|v_{\tau} - \Delta v| \leq C_0(1 + |v|^p),
\]
and satisfies
\[
|v(\xi, \tau)| \leq \eta_0(T - t)^{-1/(p-1)},
\]
for all $(\xi, \tau) \in B(a, R) \times [T - R^2, T)$ for some $a \in \mathbb{R}$ and $R > 0$, then $v$ does not blow up at $(a, T)$.
\end{proposition}
\textbf{Proof.} See Theorem 2.1 page 850 in \cite{GKcpam85}. $\square$

Since $r\geq\frac{\vep_0}{4}$, the estimate
\[
\left|(T - t)^{1/(p-1)} u(r, t) - f\left(\frac{r-r_{max}}{ \sqrt{(T-t)\log(T - t)}}\right)\right| \leq \frac{C(A)}{\sqrt{|\log(T - t)|}}.
\]
together with Proposition \ref{threshold of non-blow-up} concludes that $r$ is not a blow-up point.

\section{Reduction to a finite-dimensional problem}\label{Section-reduction-to-finite-dimension}
%Here we prove Lemma?? 
Since the definition of the shrinking set $\mathcal{S}$, given by Definition \ref{Def: shrinking set}, shows two different types of estimates, in the blow-up region and regular region, accordingly, we need two different approaches to handle those estimates:
\begin{itemize}
\item In the blow-up region, we work in similarity variables \eqref{self-similar-variables}, in particular, we crucially use the projection of equation \eqref{q's equation} with respect to the decomposition given in \eqref{decomposition of q}.

\item In the regular region, we directly work in the variables $u(x, t)$, using standard parabolic estimates.
\end{itemize}

%\subsection{Details on the dynamics of the linearized solution}
\subsection{Estimates in the blow-up region}\label{Proof-a-priori-estimates}

\textbf{Proof of Proposition \ref{Prop:control of q1,q2,q-,qe}}\\
%\textbf{section A priori estimates}\\
In this section, we prove Proposition \ref{Prop:control of q1,q2,q-,qe}. More precisely, we project the linearized equation \eqref{q's equation} on the Hermite polynomials to get the equations
satisfied by the different coordinates of the decomposition \eqref{decomposition of q}. \\

In the following, we will find the main contribution in the projections $P_i$ (for $0\leq i\leq 2$) and $P_-$
of the different terms appearing in equation \eqref{q's equation}. More precisely, the proof will be
carried out in two parts;
\begin{itemize}
\item  In the first subsection, we write equations satisfied by $q_j$, for $0\leq j\leq 2$, and $q_-$. Then, we prove (1) and (2) of Proposition 4.5.
\item  In the second subsection, we first derive from equation \eqref{q's equation} the equation satisfied by $q_e$ and prove the last identity in (3) of Proposition \ref{Prop:control of q1,q2,q-,qe}.
\end{itemize}

\textbf{Part 1: Proof of items (1) and (2) from Proposition \ref{Prop:control of q1,q2,q-,qe}}

\medskip

\noindent\textbf{First term $\partial_s q$}\\
Let $P_i$ and $P_-$ defined as in \eqref{eq: Def of projections P_m} and \eqref{eq: Def of projection P_-}, then the following holds:
\begin{equation}
\begin{aligned}
P_i (\del_s q)&= \del_s q_i \mbox{ with } i \in \{0,1,2\}, \\
P_-(\del_s q)&= \del_s q_-.
\end{aligned}
\end{equation}
%	\end{lemma}
\noindent\textbf{Second term $\mathcal{L} q$}\\
By the definition of $h_i$ given by \eqref{eq:definition of hm}, we easily obtain the projection of $\mathcal{L} q$ as follows
\begin{lemma}\label{second term}
Let $P_i$ and $P_-$ defined as in \eqref{eq: Def of projections P_m} and \eqref{eq: Def of projection P_-}, then the following holds:
\begin{equation}
\begin{aligned}
P_i (\mathcal{L} q)&= \left (1-\frac{i}{2}\right ) q_i,\mbox{ for }i=0,1,2. \\
%P_1 (\mathcal{L} q)&=  \frac{1}{2} q_1,\\
%P_2(\mathcal{L} q)&= 0,\\
P_-(\mathcal{L} q)&= \mathcal{L} q_-.
\end{aligned}
\end{equation}
\end{lemma}
\noindent\textbf{Third term $Vq$}\\

\begin{lemma}\label{Third term}
For all $A>0$, there exists an \(s_0\geq 0\) such that for all \(s\geq s_{0}\), if \(q(s)\in \mathcal V(s)\), the following estimations holds: 
\begin{equation}
\begin{aligned}
P_i(Vq)&\leq  ACs^{-2}   \quad i= 0\mbox{ or }1,\\
\left|P_2(Vq)+ {2}s^{-1}q_2\right|&\leq   2As^{-3},\\
P_-(Vq)&\leq   CA^2s^{-3}log(s)(1+|y|^3).
\end{aligned}
\end{equation}
\end{lemma}
\begin{proof}
Let us recall that \(V=p\varphi^{p-1}-\frac{p}{p-1}\), then using Taylor expansion for  $\{|y|\leq \sqrt{s}\}$, we obtain:
\begin{equation}\label{eq: expansion V}
\begin{aligned}
V&=p\left ((p-1+b\frac{y^2}{s})^{-1/(p-1)}+\frac{\kappa}{2ps}\right )^{p-1}-\frac{p}{p-1},\\
&= \frac{1}{2s}-\frac{bpy^2}{s(p-1)^2}+O(\frac{y^4}{s^2}).\\
\end{aligned}
\end{equation}
Using the fact that $q \in \mathcal{V}(s)$, $q_0$ and $q_1$ are then controlled by $s^{-2}$. Therefore,
\begin{equation}
\begin{array}{ll}
|P_{0}(Vq)|= |\int_{\Rbb}Vq h_0d\mu|,&\\
\leq  |\int_{|y|\leq \sqrt{s}} C\frac{y^2}{s}(\sum_{i=0}^2 q_ih_i+ q_-) h_0  d\mu|
+  |\int_{|y|> \sqrt{s}} C\frac{y^2}{s}(\sum_{i=0}^2 q_ih_i+ q_-) h_0  d\mu|,&\\ 
\leq C|\int_{|y|\leq \sqrt{s}} (\sum_{i=0}^2 q_ih_i+ q_-) h_0 d\mu|
+|\int_{|y|> \sqrt{s}} C\frac{y^2}{s}(\sum_{i=0}^2 q_ih_i+ q_-) h_0  \frac{e^{-\frac{y^2}{4}}}{(4\pi)^{1/2}}dy|,&\\
\leq C|q_0| +|\int_{|y|> \sqrt{s}} C\frac{y^2}{s}(\sum_{i=0}^2 q_ih_i+ q_-) h_0  \frac{e^{-\frac{y^2}{4}}}{(4\pi)^{1/2}}dy|.&\\
\end{array}
\end{equation}
Notice that $q(s)$ is in $\mathcal{V}(s)$, then by Definition \ref{Def: shrinking set} and Proposition \ref{Prop: Priori estimate}, we obtain that:
\begin{equation}
|P_0(Vq)|\leq \frac{AC}{s^2}+Ce^{-\frac{s}{8}}\leq \frac{AC}{s^2}.
\end{equation}
This is exactly the desired result for $P_0(Vq)$. The proof for $P_1(Vq)\leq e^{-\frac{Ks^2}{2}}$ is parallel to above; hence, we omit it. Using \eqref{eq: expansion V} and arguing as above, we obtain: 
\begin{equation}
\begin{aligned}
|P_{2}(Vq)+2s^{-1}q_2|
&\leq  |\int_{\Rbb}-\frac{bp}{(p-1)^2}(s^{-1})(y^2)(\sum_{i=0}^2 q_ih_i+ q_-) h_2 d\mu + 2s^{-1}q_2|\\
&\leq  ACs^{-3}.
\end{aligned}
\end{equation}
The above implies that 
\begin{equation}
P_-(Vq)=\left|Vq-\sum_{i=0}^2 P_{i}(Vq)\right|\leq CA^2s^{-3}\log(s)(1+|y|^3).
\end{equation}
\end{proof}

\noindent\textbf{Fourth term $R(y,s)$}

\begin{lemma}[Estimates for term $R$]\label{Lemma:esti:R}
For $i \leq 1$ 
\begin{equation}\label{control of $R$ on H+}
|P_i(R)|\leq Cs^{-2},
\end{equation}

\begin{equation}\label{control of $R$ on H2}
|P_2(R)|\leq Cs^{-3},
\end{equation}
\end{lemma}
\noindent and we have also:
\begin{equation}
|P_-(R)|\leq Cs^{-2}(1+|y|^{3})
\end{equation}
\begin{proof}
To give the estimates on $R$, we first compute each term in $R$ for $|y|<\sqrt{s}$ with Taylor expansion. To have a better vision of this, we remind the readers that
\[R(y,s)= \del^2_y \varphi-\frac{1}{2}y\del_y \varphi-\frac{1}{p-1}\varphi+\varphi^p-\del_s \varphi,\]
with $\varphi(y,s)=\left [(p-1+b\frac{y^2}{s})^{-\frac{1}{p-1}}+\frac{a}{s}\right ]$.

%\textcolor{red}{
%To have a better vision how orders cancel in your expansion, you must write
%\[\varphi(y,s)=\left [(p-1+b\frac{y^2}{s})^{-\frac{1}{p-1}}+\frac{a}{s}\right ],\]
%it is better to keep the constant $b$ in your expansions, then in the end you will replace by %$b=\frac{p-1}{4 p}$. \\
%Then you can start by writing 
%\[\varphi_{y}(y,s) =-\frac{2b}{(p-1)s}y\left (p-1+b\frac{y^2}{s}\right )^{-\frac{p}{p-1}}\]
%\[\varphi_{yy}(y,s) =-\frac{2b}{(p-1)s}\left (p-1+b\frac{y^2}{s}\right )^{-\frac{p}{p-1}}+ \frac{4b^2}{(p-1)^2s^2}y^2\left (p-1+b\frac{y^2}{s}\right )^{-\frac{2p-1}{p-1}}\]
%}

\medskip

We note that for$ |y|<\sqrt{s}$, $ f $ is bounded. By Taylor expansion, we obtain the following 
\begin{equation}\label{eq: R taylor}
\begin{aligned}
R(y,s)&=\left(a-\frac{wb\kappa}{(p-1)^2}\right)\frac{1}{s}+O\left(\frac{1}{s^2}\right)\\
&+\left(-\frac{abp}{(p-1)^{2}}+\frac{b\kappa}{(p-1)^2}\left(\frac{6bp}{(p-1)^2}-1\right)\right)\frac{y^2}{s^2}+O \left (\frac{y^6}{s^3} \right ).\\
\end{aligned}
\end{equation}
From the above Taylor expansion, one can easily see that:
\begin{equation}
\begin{aligned}\label{eq: P0R}
P_0(R)&=\left(a-\frac{2b\kappa}{(p-1)^2}\right)\frac{1}{s}+O\left (\frac{1}{s^2}\right ),\\
\end{aligned}
\end{equation}
and 
\begin{equation}\label{eq:P2R}
\begin{aligned}
P_2(R)= & 2\left(-\frac{abp}{(p-1)^{2}}+\frac{b\kappa}{(p-1)^2}\left(\frac{6bp}{(p-1)^2}-1\right)\right)\frac{1}{s^2}+O\left (\frac{1}{s^3}\right ).
\end{aligned}
\end{equation}
Together with our choice of $a$, $b$: 
\begin{equation}
%\left\{
%\begin{aligned}
a=\frac{(p-1)^{-\frac{1}{p-1}}}{2p},\;\;
b=\frac{(p-1)^2}{4p},\\
%\end{aligned}
%\right.
\end{equation}
we therefore obtain  $P_0(R)=O\left(\frac{1}{s^2}\right)$ and $P_2(R)=O\left(\frac{1}{s^3}\right)$.
%\begin{equation}
%	\begin{aligned}
%		P_0(R)&=O\left(\frac{1}{s^2}\right).\\
%	\end{aligned}
%\end{equation}
%\begin{equation}
%	\begin{aligned}
%		P_2(R)&=O\left(\frac{1}{s^3}\right).\\
%	\end{aligned}
%\end{equation}
The estimation for $P_1(R)$ can be argued as in Corollary 5.13 and Lemma 5.18 in \cite{MZjfa08}; hence, we omit here.

Using \eqref{eq: Def of projection P_-},\eqref{eq: R taylor},\eqref{eq: P0R}, and \eqref{eq:P2R}, we obtain the following:
$$
|P_{-}(R)|\leq \left|R-\sum_{i=0}^{2}P_i(R)\right|\leq Cs^{-2}(1+|y|^3).
$$
\end{proof}

\noindent\textbf{Fifth term $B$}:

For the quadratic term $B$, we first remind the readers of the following Lemma:
\begin{lemma}
For all $A > 0$, there exists $s_{0}\geq 0$ such that for all $\tau \geq s_{0}$, if $q(\tau) \in \mathcal{V}(\tau)$, then
\begin{equation}\label{eq:B inner estimation}
|\chi_c(y, \tau) B(q(y, \tau))| \leq C |q|^2, 
\end{equation}
and
\begin{equation}\label{eq:B outer estimation}
|B(q)| \leq C |q|^{\bar{p}} ,
\end{equation} 
where $\bar{p} = \min(p, 2)$.
\end{lemma}
\begin{proof}
This Lemma was argued in Lemma 3.6 of \cite{MZdm97}; interested readers are invited to read the proof in \cite{MZdm97}.
\end{proof}
Then, we are able to claim the following lemma:
\begin{lemma}\label{fifthterm}
There exists $s_0\geq 0$ such that if $q(s)\in \mathcal{V}(s)$ for $s>s_0$, then $B$ verifies:
\begin{equation}
\begin{aligned}
P_i(B)&\leq CA^2s^{-3}, \quad i\in \{0,1,2\}\\
%B_2 &\leq Cs^{-3},\\
P_{-}(B)& \leq CAs^{-2} (1+|y|^3).\\
\end{aligned}
\end{equation}
\end{lemma}
\begin{proof}
We argue as in the proof of Lemma 5.10 and Lemma 5.17 in \cite{MZjfa08}.
\end{proof}
\noindent\textbf{Sixth term $H$:}

\begin{lemma}\label{sixth term}
The following estimations holds: 
\begin{equation}
\begin{aligned}
P_i(H)&\leq  Ce^{-s/2}   \quad i= 0,1\mbox{ or }2,\\
%P_2(H)&\leq    Ce^{-s/2},\\
P_-(H)&\leq    Ce^{-s/2}(1+|y|^3),
\end{aligned}
\end{equation}
\end{lemma}
\begin{proof} We argue it as in the proof of Lemma 3.9 from  \cite{MNZNon2016}.
\end{proof}
\noindent\textbf{Seventh term $\del_y G$}:
\begin{lemma}\label{seventh term}
For $\del_y G$, we have the following estimations: 
\begin{equation}
\begin{aligned}
P_i(\del_y G)&\leq  Ce^{-s/2}   \quad i= 0, 1\mbox{ or }2,\\
%			P_2(\del_y G)&\leq    Ce^{-s/2},\\
P_-(\del_y G)&\leq    Ce^{-s/2}(1+|y|^3),
\end{aligned}
\end{equation}
\end{lemma}
\begin{proof}
This can be done with integration by parts, interested readers are invited to see the proof of Lemma 5.19 in \cite{MNZNon2016}. 
\end{proof}
\noindent\textbf{Eighth term $N$: }
\begin{lemma}[projection of the last term: $N$ ]\label{Eighth term}

\begin{equation}\label{control of $N$ on H+}
|P_i(N)|\leq Ce^{-s}\mbox{ where }i=0,1\mbox{ or }2.
\end{equation}
and 
\begin{equation}\label{control of $N$ on H-}
|P_-(N)|\leq e^{-s/2}(1+|y|^3).
\end{equation}
\end{lemma}
\begin{proof}
%The projection of each term on the eigenfunctions is as follows:
%\begin{equation}
%\begin{split}
%P_0(N)&=\int_{\Rbb}N\frac{h_0}{\|h_0\|_{L_\rho^2}}\rho(y)dy\\
%&= e^{-s/2}\int_{\Rbb}\frac{d-1}{ye^{-s/2}+1}W\del_y\Xcal\frac{e^{-\frac{|y|^2}{4}}}{(4\pi)^{1/2}}dy
%\end{split}
%\end{equation}
%\begin{equation}
%\begin{aligned}
%P_1(N)&=\int_{\Rbb}N\frac{h_1}{\|h_1\|_{L_\rho^2}}\rho(y)dy\\
%&=\frac{e^{-s/2}}{2}\int_{\Rbb}\frac{d-1}{ye^{-s/2}+1}W\del_y\Xcal \frac{e^{-\frac{|y|^2}{4}}}{(4\pi)^{1/2}}ydy
%\end{aligned}
%\end{equation}
%\begin{equation}
%\begin{aligned}
%P_2(N)&=\int_{\Rbb}N\frac{h_2}{\|h_2\|_{L_\rho^2}}\rho(y)dy\\
%&=\frac{e^{-s/2}}{8}\int_{\Rbb}\frac{d-1}{ye^{-s/2}+1}W\del_y\Xcal \frac{e^{-\frac{|y|^2}{4}}}{(4\pi)^{1/2}}(y^2-2)dy
%\end{aligned}
%\end{equation}
Let us first recall that 
\[N(y,s)= \frac{d-1}{y+e^{s/2}}W\del_y\chi,\]
where $W$ is defined by \eqref{eq: W} and $\chi$ is the cut-off function defined in \eqref{eq: def Chi}.\\
We will now give the estimation on the terms $P_i(N)$, $i\in \lbrace0,1,2\rbrace$. 
From Lemma \ref{Lemma: est. W}, we have that for $|y|\geq e^{\frac{s}{2}}(\frac{1}{8}\vep_0-1)$, $\|W(s)\|_{L^\infty} \leq \kappa + 2$. By the definition \eqref{eq: def Chi} of $\chi$,  we easily have that
\begin{equation}\label{eq:estim dely Chi}
|\del_y \mathcal{X}|\leq e^{-s/2}\frac{C}{\vep_0}\mathbb{I}_{\{(\frac{1}{8}\vep_0-1)e^{s/2}\leq y \leq (\frac{1}{4}\vep_0-1)e^{s/2}\}}.
\end{equation}
Using the estimations above, we get
\begin{equation}\label{eq:estim:abs:eta0}
\begin{aligned}
|P_0(N)| 
&\leq e^{-s/2}\int_{\Rbb}\left|\frac{d-1}{ye^{-s/2}+1}W\del_y\Xcal\right|\frac{e^{-\frac{|y|^2}{4}}}{(4\pi)^{1/2}}dy,\\
&\leq C e^{-s} \int_{\{y\geq (\frac{1}{8}\vep_0-1)e^{s/2}\}}\left|\frac{d-1}{ye^{-s/2}+1}\right|\frac{e^{-\frac{|y|^2}{4}}}{(4\pi)^{1/2}}dy,\\
&\leq C e^{-s} \int_{\{y\geq (\frac{1}{8}\vep_0-1)e^{s/2}\}}\frac{e^{-\frac{|y|^2}{4}}}{(4\pi)^{1/2}}dy\leq C e^{-s}.
\end{aligned}
\end{equation}
Arguing in a similar fashion, we obtain the desired estimation for $P_1(N)$ and $P_2(N)$.

%with a similar idea,
%we have: 
%\begin{equation}\label{eq:estim:abs:eta1 et 2}
%\begin{aligned}
%|P_1(N)| &\leq e^{-s} (\kappa+2)\frac{C}{2\vep_0}\int_{y\geq (\frac{3}{4}\vep_0-1)e^{s/2}\}}\left|\frac{y(d-1)}{ye^{-s/2}+1}\right|\frac{e^{-\frac{|y|^2}{4}}}{(4\pi)^{1/2}}dy\\
%&\leq e^{-s} (\kappa+2)\frac{2C}{3\vep_0^2}\int_{\{y\geq (\frac{3}{4}\vep_0-1)e^{s/2}\}}\left|y\right|\frac{e^{-\frac{|y|^2}{4}}}{(4\pi)^{1/2}}dy\\
%|P_2(N)| &\leq e^{-s} (\kappa+2)\frac{C}{8\vep_0}\int_{\{y\geq (\frac{3}{4}\vep_0-1)e^{s/2}\}}\left|\frac{(d-1)}{ye^{-s/2}+1}\right|\frac{e^{-\frac{|y|^2}{4}}}{(4\pi)^{1/2}}\left|y^2-2
%\right|dy\\
% &\leq e^{-s} (\kappa+2)\frac{C}{6\vep_0^2}\int_{\{y\geq (\frac{3}{4}\vep_0-1)e^{s/2}\}}\frac{e^{-\frac{|y|^2}{4}}}{(4\pi)^{1/2}}\left|y^2-2.
%\right|dy\\
%\end{aligned}
%\end{equation}
We conclude the proof with the estimation of $P_-(N)$. Using Lemma \ref{Lemma: est. W}, \eqref{control of $N$ on H+} and \eqref{eq:estim dely Chi}, we obtain,
\begin{equation}
|P_{-}(N)|= |N-\sum_{i=0}^2P_i(N)h_i|\leq e^{-s/2}(1+|y|^3).
\end{equation}
\end{proof}

\noindent\textbf{Proof of Proposition \ref{Prop:control of q1,q2,q-,qe}}\\
\textit{Proof of item (1) and (2) of Proposition \ref{Prop:control of q1,q2,q-,qe}}

Using Lemma \ref{second term}, Lemma \ref{Third term}, Lemma \ref{Lemma:esti:R}, Lemma \ref{fifthterm}, Lemma \ref{sixth term}, Lemma \ref{seventh term} and Lemma \ref{Eighth term} and arguing as the proof of Proposition 4.6 in  \cite{MZjfa08}, we can easily obtain 
\[
\begin{array}{l}
\left | q'_0(s)-q_0(s)\right |\leq \frac{AC}{s^{2}} \mbox{ and  }\left | q'_1(s)-\frac{1}{2}q_1(s)\right |\leq \frac{AC}{s^{2}},
\end{array}
\]
this concludes (1) from Proposition \ref{Prop:control of q1,q2,q-,qe}. 

The case (2) is more delicate. From Lemma \ref{second term}, Lemma \ref{Third term}, Lemma \ref{Lemma:esti:R}, Lemma \ref{fifthterm}, Lemma \ref{sixth term}, Lemma \ref{seventh term}, and Lemma \ref{Eighth term}, we obtain

\begin{equation}
\left|q_2'(s)+\frac{2}{s}q_2(s)\right|\leq C\frac{A^2}{s^{3}}.
\end{equation}
Integrating this inequality between $\tau$ and $s$ gives the desired estimates on $q_2$,

\begin{equation}
\left|q_2(s)\right|\leq \left(\frac{\tau}{s}\right)^2q_2{\tau} + C\frac{A^2}{s^{2}}\log(s/\tau).
\end{equation}

For $q_-$, we can use the properties of the semi-group generated by $\mathcal{L}$, and obtain that for all $s\in [\tau,s_1]$,

$$
\begin{aligned}
q_-(s)&=e^{(s-\tau)\mathcal{L}}q_{-}(\tau)\\
&+\int_\tau ^s e^{(s-s')\mathcal{L}} P_-(Vq+H(y,s)+\del_y G(y,s)+R(y,s)+B(y,s)+N(y,s))ds'.\\
\end{aligned}
$$
Arguing as Lemma A.2 in \cite{MZjfa08} gives us: 
\[
\begin{aligned}
&\left\|\frac{q_-(s)}{1+|y|^3}\right\|_{L^{\infty}}=e^{-\frac{3}{2}(s-\tau)}\left\|\frac{q_{-}(\tau)}{1+|y|^3}\right\|_{L^{\infty}}\\
&+\int_\tau ^se^{-\frac{3}{2}(s-s')}\left\|\frac{P_-(Vq+H(y,s)+\del_y G(y,s)+R(y,s')+B(y,s')+N(y,s'))}{1+|y|^3}\right\|_{L^{\infty}}ds'.\\
\end{aligned}
\]
Assuming that $q(s')\in V_A(s')$, the estimations Lemma \ref{second term}, Lemma \ref{Third term}, Lemma \ref{Lemma:esti:R}, Lemma \ref{fifthterm}, Lemma \ref{sixth term}, Lemma \ref{seventh term}, and Lemma \ref{Eighth term} imply the following
\[
\begin{aligned}
\left\|\frac{q_-(s)}{1+|y|^3}\right\|_{L^{\infty}}&=e^{-\frac{3}{2}(s-\tau)}\left\|\frac{q_{-}(\tau)}{1+|y|^3}\right\|_{L^{\infty}}\\
&+\int_\tau ^s e^{-\frac{3}{2}(s-s')}
\left[\frac{A^2}{s'^3}\log(s')+\frac{C}{s'^2}+\frac{CA}{s'^2}+Ce^{-\frac{s'}{2}}\right]ds'.\\
\end{aligned}
\]
Using Gronwall's Lemma we deduce that:
\[
\begin{aligned}
e^{\frac{3}{2}s}\left\|\frac{q_-(s)}{1+|y|^3}\right\|_{L^{\infty}}&=e^{-\frac{3}{4}(s-\tau)}e^{\frac{3}{2}(\tau)}\left\|\frac{q_{-}(\tau)}{1+|y|^3}\right\|_{L^{\infty}}\\
&+e^{\frac{3}{2}s} 2^{\frac{5}{2}}\left[\frac{A^2}{s^3}\log(s)+\frac{C}{s^2}+\frac{CA}{s^2}+Ce^{-\frac{s}{2}}\right].\\
\end{aligned}
\]
This concludes the estimation on $P_-(q)$.
\medskip

%\noindent\textit{Proof of $q_e$ of Proposition \ref{Prop:control of q1,q2,q-,qe}}
\textit{}

\textbf{Part 2: The outer region $q_e$: proof of item (3) from Proposition \ref{Prop:control of q1,q2,q-,qe}}
%\textcolor{red}{You will proceed as in the proof of 5.3 The outer region : qe from Masmoudi and Zaag \cite{MZjfa08}}
%\medskip

Here, we conclude the proof of Proposition \ref{Prop:control of q1,q2,q-,qe} by demonstrating the final inequality about $q_e$. As $q(s)\in \mathcal  V(s)$ for all $s \in [\tau, s_1]$, it follows that
\[
\|q(s)\|_{L^{\infty}(|y|<2K\sqrt{s})} \leq \frac{C A^2}{\sqrt{s}}.
\]
We note that terms $H$, $\del_y G$, and $N$ defined in \eqref{definition of B,R,F,N} are compactly supported in $\left [(\frac{3}{8}\vep_0-1)e^{s/2}, (\frac{3}{4}\vep_0-1)e^{s/2}\right ]$. Then, we derive that the equation satisfied by $q_e$ is 
%\textcolor{red}{I think that terms $H$, $\pa_y G$ and $N$ will appear in the equation of $q_e$}
\begin{equation}\label{qe equation }
\begin{aligned}
\del_s q_e = &(\mathcal{L}+V)q_e-q(s)\left (\del_s \chi_c+\Delta \chi_c+\frac{1}{2}y\nabla\chi_c\right )\\
&+\left(R(y,s)+B(y,s)+H(y,s)+\del_yG(y,s)+N(y,s)\right)(1-\chi_c)+2 div(q(s)\nabla \chi_c).
\end{aligned}
\end{equation}
Writing this equation in its integral form and using the maximum principle satisfied by $e^{\tau\mathcal{L}}$, we deduce that
\begin{equation}
\begin{aligned}
\|q_e\|_{L^{\infty}}&\leq e^{-\frac{s-\tau}{p-1}} \|q_e(\tau)\|_{L^{\infty}},\\
&+\int_{\tau}^se^{-\frac{s-s'}{p-1}}\|\left((1-\chi_c)(R+B+H+\del_y G+N)\right)\|_{L^{\infty}}ds',\\
&+\int_{\tau}^se^{-\frac{s-s'}{p-1}}\|q(s')\left (\del_s \chi_c+\Delta \chi_c+\frac{1}{2}y\nabla\chi_c\right )\|_{L^{\infty}}ds',\\
&+\int_{s}^{\tau} e^{-\frac{s-s'}{p-1}} \left(\frac{1}{\sqrt{1 - e^{-(s-s')}}}\right) \left\| q(s') \nabla \chi_c \right\|_{L^{\infty}} ds'.
\\
\end{aligned}
\end{equation}
Notice that with Lemma \ref{Lemma:esti:R}, Lemma \ref{fifthterm}, Proposition \ref{Prop: Priori estimate}, and \eqref{def_chi_c}, by arguing as Section 5.3 in \cite{MZjfa08}, we obtain the following bound:
\begin{equation}\label{eq: esti qe integral terms}
\begin{array}{ll}
\|q(s')\left (\del_s \chi_c+\Delta \chi_c+\frac{1}{2}y\nabla\chi_c\right )\|_{L^{\infty}}&\leq C\frac{A^2}{\sqrt{s'}},\\
\left\| q(s') \nabla \chi_c \right\|_{L^{\infty}}&\leq C\frac{A^2}{s'},\\
\|(1-\chi_c)R(y,s')\|_{L^\infty}&\leq \frac{C}{s'},\\
\|(1-\chi_c)B(y,s')\|_{L^\infty}&\leq \frac{1}{2(p-1)}\|q_e\|_{L^\infty}.
\end{array}
\end{equation}
Then, with Lemma \ref{Lemma: est. W} and \eqref{eq:estim dely Chi}, we gain the following
\begin{equation}
(1-\chi_c)(H(y,s')+\del_yG(y,s')+N(y,s'))\leq Ce^{-\frac{s'}{2}}\leq \frac{C}{s'}
\end{equation}
By choosing $K$ large enough such that estimations \eqref{eq: esti qe integral terms} are verified, we write
\begin{equation}
\begin{array}{ll}
\|q_e\|_{L^{\infty}}&\leq e^{-\frac{s-\tau}{p-1}} \|q_e(\tau)\|_{L^{\infty}},\\
&+\displaystyle\int_{\tau}^se^{-\frac{s-s'}{p-1}}\left(\frac{1}{2(p-1)}\|q_e(s')\|_{L^\infty}+\frac{CA^2}{\sqrt{s'}}+\frac{A^2}{s'}\frac{1}{\sqrt{1-e^{-(s-s')}}}\right)ds',\\
\\
\end{array}
\end{equation}
We then conclude with Gronwall's inequality
\[
\|q_e(s)\|_{L^\infty}\leq e^{-\frac{(s-\tau)}{2(p-1)}}
\|q_e(\tau )\|_{L^\infty}+C\frac{A^2}{\sqrt{\tau}}(1+s-\tau). 	
\]

\subsection{Estimates in the regular region}\label{section_regular_region}
Our goal here is to show that
\begin{equation}\label{eq:estimation on u in regularregion}
|x|\leq \frac{3\vep_0}{4}, \text{ then we have } u(x,t^*) \leq \frac{\eta_0}{2}.
\end{equation}
This is shown in three steps:
\begin{itemize}
\item In the first step, we improve the bounds on the solution $u(x,t)$ in the intermediate region.
\item In the second step, we use parabolic regularity to obtain an estimation of the solution in the region $\mathcal{R}_2$.
\item Finally, we use the two steps above to get \eqref{eq:estimation on u in regularregion}
\end{itemize}

\noindent\textbf{Step 1: Improved estimates in the intermediate region}

Here, we refine the estimates on the solution in the following intermediate region:

\begin{equation}\label{def: intermediate region}
\frac{\vep_0}{8}\leq|x|\leq K\sqrt{(T-t)\log(T-t)}.
\end{equation}
By Lemma \ref{Lemma: est. W}, we have
\begin{equation}
\forall t \in [0, t^*], \text{ and } \forall x\in \Rbb^n,  \left| u(t) \right| \leq C(T - t)^{-\frac{1}{p-1}}, \label{eq:bound intermediate region}
\end{equation}
valid in particular in the intermediate region given by  \eqref{def: intermediate region}. This bound is unsatisfactory since it goes to infinity as $t \to T$. In order to refine it, given a $x$ small enough in norm $|x|$ , we use this bound when $t = t_0(x)$ defined by
\begin{equation}
\left| x\right| = K_0 \sqrt{ (T - t_0(x)) \left| \log(T - t_0(x)) \right| }, \label{eq:t0}
\end{equation}
to see that the solution is, in fact, flat at that time. Then, advancing the PDE \eqref{equation-NLH-1d}, we see that the solution remains flat. More precisely, we claim the following:
\begin{lemma}[flatness of the solution in the Intermediate region in \eqref{def: intermediate region}]\label{Lemma: flateness}
There exists $\zeta_0 > 0$ such that for all $K > 0$, $\varepsilon_0 > 0$, $A \geq 1$, there exists $s_{0,9}(K, \varepsilon_0, A)$ such that if $s_0 \geq s_{0,9}$ and $0 < \eta_0 \leq 1$, then, $\forall t_0(x) \leq t \leq t^*$,
\begin{equation}
\left| \frac{u(x, t)}{u^*(x)} - \frac{U_{K}(x)}{U_{K}(1)} \right| \leq \frac{C }{\left| \log |x| \right|^{\zeta_0}},
\end{equation}
where $u^*$ is defined in \eqref{eq: profile} and
\begin{equation}
U_{K}(\tau) = \kappa \left( (1 - \tau) + \frac{(p-1)K^2}{4p} \right)^{-\frac{1}{p-1}}. \label{eq:UK0}
\end{equation}
In particular, $\left| u(x, t) \right| \leq C(K) \left| u^*(|x|) \right|$.
\end{lemma}
\begin{proof}
See in \cite{MNZNon2016} P.316 Lemma 3.12.
\end{proof}

\noindent\textbf{Step 2: A parabolic estimate in regular region:}

Recall from the definition on $\mathcal{V}$, that:
$$\forall x\in \Rbb \mbox{ such that } 0\leq|x|\leq \frac{3\vep_0}{4}, u(x,t)\leq \eta_0 .$$
Using parabolic estimation on the solution, for $u(x,t)$ in region $\mathcal{R}_2$, we claim the following: 
\begin{proposition}\label{Prop: parabolic estimate}
For all \( \varepsilon > 0 \), \( \varepsilon_0 > 0 \), \( \sigma_1 \geq 0 \), there exists \( T \geq 0 \) such that for all \( \overline{t} \leq T \), if \( u \) is a  solution to
\[
\partial_t u = \Delta u + |u|^{p-1}u \quad \text{for all } x \in [0,3\vep_0/4], t \in [0, \overline{t}],
\]
which satisfies:
\begin{enumerate}
\item[(i)] For \( |x| \in [\frac{\varepsilon_0}{8}, \frac{3\varepsilon_0}{4}] \), \( |u(x, t)| \leq \sigma_1 \).
\item[(ii)] For \( 0 \leq |x| \leq \frac{\vep_0}{8} \), \( u(x, 0) = 0 \).
\end{enumerate}
Then, for all \( t \) in \( [0, \overline{t}] \), for all \(  |x| \leq \frac{3\vep_0}{4} \), $\left|u(x, t)\right| \leq \varepsilon$.
\end{proposition}
\begin{proof}
Consider $\overline{u}$, recalled here, after a trivial chain rule to transform the $\partial_x u$ term:
\[
\forall t \in [0, \overline{t}], \quad \forall x \in \mathbb{R}, \quad \partial_{t}\overline{ u} = \Delta \overline{u} + |\overline{u}|^{p-1} \overline{u} - 2\nabla (\overline{\chi}' u) + \overline{\chi}'' u.
\]
Therefore, since $\overline{u}(x, 0) \equiv 0$, we write
\[
\|\overline{u}(t)\|_{L^\infty} \leq \int_0^t S(t - t')\left(|u|^{p-1} I_{|x| \leq \frac{\varepsilon_0}{4}} \overline{u} - 2\nabla \left(\overline{\chi}' u I_{|x| \leq \frac{\varepsilon_0}{4}}\right)\right) + \overline{\chi}'' u(t') I_{|x| \leq \frac{\varepsilon_0}{4}} \, dt'.
\]
where $S(t)$ is the heat kernel. Since $\overline{\chi}'$ and $\overline{\chi}''$ are supported by $\left\{\frac{3\varepsilon_0}{8} \leq |x| \leq \frac{3\varepsilon_0}{4}\right\}$ and satisfy $|\overline{\chi}'| \leq \frac{C}{\varepsilon_0}$, $|\overline{\chi}''| \leq \frac{C}{\varepsilon_0^2}$ and using parabolic regularity, we write
\[
\|\overline{u}(t)\|_{L^\infty} \leq \sigma_1^{p-1} \int_0^t \|\overline{u}(t')\| \, dt' + C\sigma_1 \frac{1}{\varepsilon_0} \int_0^t\frac{1}{ \sqrt{t - t'}} dt' + C\sigma_1 \frac{1}{\varepsilon_0^2} \int_0^t dt'.
\]
If $\overline{t} < 1$, by Gronwall's estimate, this implies that
\[
\|\overline{u}(t)\|_{L^\infty} \leq Ce^{\sigma_1^{p-1}} \left( \frac{\sigma_1}{\varepsilon_0} \sqrt{\overline{t}} +  \frac{\sigma_1}{\varepsilon_0^2} \overline{t}\right).
\]
Taking $\overline{t}$ small enough, we can obtain $\forall t \in [0,\overline{t}]$, $\|\overline{u}(t)\|_{L^\infty} \leq \varepsilon$.

\end{proof}

\textbf{Step 3}: \textbf{Proof of the improvement in Definition \ref{Def: shrinking set}}

Here, we use Step 1 and Step 2 to prove \eqref{eq:estimation on u in regularregion}, for a suitable choice of parameters.

Let us consider $K > 0$ defined in Lemma \ref{Lemma: flateness} and $\delta_0(K) > 0$. Then, we consider $\varepsilon_0 \leq 2\delta_0$, $0 < \eta_0 \leq 1$ defined in Lemma \ref{Lemma: flateness} and Proposition \ref{Prop: parabolic estimate}; $A \geq 1$, $s_0$ sufficiently large such that conditions in \eqref{open condition bootstrap} and Lemma \ref{Lemma: flateness} and Proposition \ref{Prop: parabolic estimate} hold.

Applying Lemma \ref{Lemma: flateness}, we see that for all  $|x| \leq \delta_0$, $A \geq 1$, for all $t \in [0, t^*]$, $|u(x, t)| \leq C(K)|u^*(x)|$.

%In particular, for all $\delta_0 \leq \frac{ 3\varepsilon_0}{8} \leq |x| \leq \frac{3\varepsilon_0}{4} $, for all $t \in [0, t^*]$, $|u(|x|, t)| \leq C(K)|u^*(\frac{3\varepsilon_0}{8})|$.

In particular, for all $\frac{ \varepsilon_0}{8} \leq |x| \leq \frac{3\varepsilon_0}{4} \leq \delta_0$, for all $t \in [0, t^*]$, $|u(|x|, t)| \leq C(K)|u^*(\frac{\varepsilon_0}{8})|$.

Using item (iii) of Lemma \ref{Lem: preparation Inditial data}, we see that for all $0\leq |x| \leq  \frac{\varepsilon_0}{8} $, $u(|x|, 0) = 0$.

Therefore, Proposition \ref{Prop: parabolic estimate} applies with $\varepsilon = \frac{\eta_0}{2}$ and $\sigma_1 = C(K) u^*(\frac{\varepsilon_0}{8})$, and we see that for all $|x| \leq \frac{3\vep_0}{4}$, for all $t \in [0, t^*]$, $|u(|x|, t)| \leq \frac{\eta_0}{2}$ and estimate $\eqref{eq:estimation on u in regularregion}$ holds.

%\begin{appendix}
\section{Appendix}

\begin{lemma}\label{Lemma: est. W}
For all $K_0 > 0$, $\varepsilon_0 > 0$, $A \geq 1$, there exists $s_0 $ such that if $s \geq s_0$, $0 < \eta_0 \leq 1$, and we assume that $u(t) \in \mathcal{S}(t)$ defined in Definition 3.1, where $t = T - e^{-s}$, then we have:

$$\|W(.,s)\|_{L^\infty} \leq \kappa + 2.$$

\end{lemma}
\begin{proof}
For $W$, we can see that:
\begin{itemize}
\item If $|y| \geq \varepsilon_0e^{s/2}(\frac{1}{4}\vep_0-1)$, then $W(y, s) = w(y, s) = \varphi(y, s) + q(y, s)$. Since $|\varphi|_{L^\infty} \leq \kappa + 1$ from \eqref{definition of varphi}, using (ii), we see that $|W|_{L^\infty} \leq \kappa + 2$ for $s$ large enough, which is for $T$ small enough.
\item If $|y| < \varepsilon_0e^{s/2}(\frac{1}{4}\vep_0-1)$, then $W(y, s) = e^{-s(p-1)}u(x e^{-s/2}, t)$ with $|x| \geq \varepsilon_0^2/2$. By (ii) of Definition 3.1, we see that $|W(y, s)| \leq \eta_0e^{-s(p-1)} \leq \eta_0T^{1/(p-1)} \leq 1$ if $\eta_0 \leq 1$ and $T \leq 1$.
\end{itemize}
\end{proof}

%\end{appendix}

%\begin{equation}
%\text{Series expansion of } (1 + by)^{-\frac{1}{p-1}}:
%+ \left(\frac{b^2}{2(p - 1)^2} + \frac{b^2}{2p - 2}\right)y^2
%- \frac{by}{p - 1} + 1
%\end{equation}
\bibliographystyle{alpha}
\bibliography{heat-20-02-2023}

\def\cprime{$'$}
\begin{thebibliography}{DDPW20}

\bibitem[BE89]{BEbook89}
J.~Bebernes and D.~Eberly.
\newblock {\em Mathematical problems from combustion theory}, volume~83 of {\em
  Applied Mathematical Sciences}.
\newblock Springer-Verlag, New York, 1989.

\bibitem[BFG10]{BFGpd10}
G.~Baruch, G.~Fibich, and N.~Gavish.
\newblock Singular standing-ring solutions of nonlinear partial differential
  equations.
\newblock {\em Phys. D}, 239(20-22):1968--1983, 2010.

\bibitem[BK88]{BKcpam88}
M.~Berger and R.~V. Kohn.
\newblock A rescaling algorithm for the numerical calculation of blowing-up
  solutions.
\newblock {\em Comm. Pure Appl. Math.}, 41(6):841--863, 1988.

\bibitem[BKL94]{BKLcpam94}
J.~Bricmont, A.~Kupiainen, and G.~Lin.
\newblock Renormalization group and asymptotics of solutions of nonlinear
  parabolic equations.
\newblock {\em Comm. Pure Appl. Math.}, 47(6):893--922, 1994.

\bibitem[Con78]{Conbook78}
C.~Conley.
\newblock {\em Isolated invariant sets and the {M}orse index}, volume~38 of
  {\em CBMS Regional Conference Series in Mathematics}.
\newblock American Mathematical Society, Providence, R.I., 1978.

\bibitem[DDPW20]{DDWIM20}
J.~D\'{a}vila, M.~Del~Pino, and J.~Wei.
\newblock Singularity formation for the two-dimensional harmonic map flow into
  {$S^2$}.
\newblock {\em Invent. Math.}, 219(2):345--466, 2020.

\bibitem[DNZ23a]{DNZMAMS20}
G.~K. Duong, N.~Nouaili, and H.~Zaag.
\newblock Construction of blow-up solutions for the complex ginzburg-landau
  equation with critical parameters.
\newblock {\em Mem. Amer. Math. Soc.}, 285, 2023.

\bibitem[DNZ23b]{DNZArxiv2022}
G.~K. Duong, N.~Nouaili, and H.~Zaag.
\newblock Modulation theory for the flat blow-up solutions of nonlinear heat
  equation.
\newblock {\em Commun. Pure Appl. Anal.}, 22(10):2925--2959, 2023.

\bibitem[GK85]{GKcpam85}
Y.~Giga and R.~V. Kohn.
\newblock Asymptotically self-similar blow-up of semilinear heat equations.
\newblock {\em Comm. Pure Appl. Math.}, 38(3):297--319, 1985.

\bibitem[GMS04]{GMSmmas04}
Y.~Giga, S.~Matsui, and S.~Sasayama.
\newblock On blow-up rate for sign-changing solutions in a convex domain.
\newblock {\em Math. Methods Appl. Sci.}, 27(15):1771--1782, 2004.

\bibitem[HV92a]{HVcpde92}
M.~A. Herrero and J.~J.~L. Vel{\'a}zquez.
\newblock Blow-up profiles in one-dimensional, semilinear parabolic problems.
\newblock {\em Comm. Partial Differential Equations}, 17(1-2):205--219, 1992.

\bibitem[HV92b]{HVCRAS92}
M.~A. Herrero and J.~J.~L. Vel\'{a}zquez.
\newblock Comportement g\'{e}n\'{e}rique au voisinage d'un point d'explosion
  pour des solutions d'\'{e}quations paraboliques unidimensionnelles.
\newblock {\em C. R. Acad. Sci. Paris S\'{e}r. I Math.}, 314(3):201--203, 1992.

\bibitem[HV92c]{HVdie92}
M.~A. Herrero and J.~J.~L. Vel{\'a}zquez.
\newblock Flat blow-up in one-dimensional semilinear heat equations.
\newblock {\em Differential Integral Equations}, 5(5):973--997, 1992.

\bibitem[HV92d]{HVasnsp92}
M.~A. Herrero and J.~J.~L. Vel{\'a}zquez.
\newblock Generic behaviour of one-dimensional blow up patterns.
\newblock {\em Ann. Scuola Norm. Sup. Pisa Cl. Sci. (4)}, 19(3):381--450, 1992.

\bibitem[Kap80]{KSJAM80}
A.~K. Kapila.
\newblock Reactive-diffusion system with arrhenius kinetics: Dynamic of
  ignition.
\newblock {\em SIAM J. Appl. Math.}, 39:21--36, 1980.

\bibitem[KP80]{KPsiam80}
D.~R. Kassoy and J.~Poland.
\newblock The thermal explosion confined by a constant temperature boundary.
  {I}. {T}he induction period solution.
\newblock {\em SIAM J. Appl. Math.}, 39(3):412--430, 1980.

\bibitem[MNZ16]{MNZNon2016}
F.~Mahmoudi, N.~Nouaili, and H.~Zaag.
\newblock Construction of a stable periodic solution to a semilinear heat
  equation with a prescribed profile.
\newblock {\em Nonlinear Anal.}, 131:300--324, 2016.

\bibitem[MRS20]{MRSIMRNI20}
F.~Merle, P.~Rapha\"{e}l, and J.~Szeftel.
\newblock On strongly anisotropic type {I} blowup.
\newblock {\em Int. Math. Res. Not. IMRN}, (2):541--606, 2020.

\bibitem[MZ97]{MZdm97}
F.~Merle and H.~Zaag.
\newblock Stability of the blow-up profile for equations of the type
  {$u_t=\Delta u+\vert u\vert ^{p-1}u$}.
\newblock {\em Duke Math. J.}, 86(1):143--195, 1997.

\bibitem[MZ08]{MZjfa08}
N.~Masmoudi and H.~Zaag.
\newblock Blow-up profile for the complex {G}inzburg-{L}andau equation.
\newblock {\em J. Funct. Anal.}, 255(7):1613--1666, 2008.

\bibitem[MZ22]{MZIMRN22}
Frank Merle and Hatem Zaag.
\newblock Behavior rigidity near non-isolated blow-up points for the semilinear
  heat equation.
\newblock {\em Int. Math. Res. Not. IMRN}, (20):16196--16260, 2022.

\bibitem[NZ18]{NZ2017}
N.~Nouaili and H.~Zaag.
\newblock Construction of a blow-up solution for the complex ginzburg-landau
  equation in some critical case.
\newblock {\em Arch. Rat. Mech. Anal}, 228(3), 2018.

\bibitem[QS07]{QSbook07}
P.~Quittner and P.~Souplet.
\newblock {\em Superlinear parabolic problems}.
\newblock Birkh{\"a}user Advanced Texts: Basler Lehrb{\"u}cher. [Birkh{\"a}user
  Advanced Texts: Basel Textbooks]. Birkh{\"a}user Verlag, Basel, 2007.
\newblock Blow-up, global existence and steady states.

\bibitem[Rap06]{Rdm06}
P.~Rapha\"{e}l.
\newblock Existence and stability of a solution blowing up on a sphere for an
  {$L^2$}-supercritical nonlinear {S}chr\"{o}dinger equation.
\newblock {\em Duke Math. J.}, 134(2):199--258, 2006.

\bibitem[RS09]{RSCMP09}
P.~Rapha\"{e}l and J.~Szeftel.
\newblock Standing ring blow up solutions to the {$N$}-dimensional quintic
  nonlinear {S}chr\"{o}dinger equation.
\newblock {\em Comm. Math. Phys.}, 290(3):973--996, 2009.

\bibitem[Vel92]{VELcpde92}
J.~J.~L. Vel{\'a}zquez.
\newblock Higher-dimensional blow up for semilinear parabolic equations.
\newblock {\em Comm. Partial Differential Equations}, 17(9-10):1567--1596,
  1992.

\bibitem[Vel93a]{VELtams93}
J.~J.~L. Vel{\'a}zquez.
\newblock Classification of singularities for blowing up solutions in higher
  dimensions.
\newblock {\em Trans. Amer. Math. Soc.}, 338(1):441--464, 1993.

\bibitem[Vel93b]{VELiumj93}
J.~J.~L. Vel{\'a}zquez.
\newblock Estimates on the {$(n-1)$}-dimensional {H}ausdorff measure of the
  blow-up set for a semilinear heat equation.
\newblock {\em Indiana Univ. Math. J.}, 42(2):445--476, 1993.

\bibitem[Zaa98]{ZAAihn98}
H.~Zaag.
\newblock Blow-up results for vector-valued nonlinear heat equations with no
  gradient structure.
\newblock {\em Ann. Inst. H. Poincar{\'e} Anal. Non Lin{\'e}aire},
  15(5):581--622, 1998.

\bibitem[Zaa02a]{ZAAaihp02}
H.~Zaag.
\newblock On the regularity of the blow-up set for semilinear heat equations.
\newblock {\em Ann. Inst. H. Poincar\'e Anal. Non Lin\'eaire}, 19(5):505--542,
  2002.

\bibitem[Zaa02b]{ZAAcmp02}
H.~Zaag.
\newblock One-dimensional behavior of singular {$N$}-dimensional solutions of
  semilinear heat equations.
\newblock {\em Comm. Math. Phys.}, 225(3):523--549, 2002.

\end{thebibliography}

\end{document}